\documentclass{amsart}
\usepackage{amssymb}
\usepackage{amsmath}
\usepackage{a4}
\usepackage{color}

\begin{document}
%\date{\version}
\newtheorem{theorem}{Theorem}[section]
\newtheorem{lemma}[theorem]{Lemma}
\newtheorem{remark}[theorem]{Remark}
\newtheorem{definition}[theorem]{Definition}
\newtheorem{corollary}[theorem]{Corollary}
\newtheorem{example}[theorem]{Example}
\def\qedbox{\hbox{$\rlap{$\sqcap$}\sqcup$}}
\makeatletter
  \renewcommand{\theequation}{%
   \thesection.\alph{equation}}
  \@addtoreset{equation}{section}
 \makeatother
\title[Sectional curvature and the Spectral Theorem]
{Sharp sectional curvature bounds and a new proof of the Spectral Theorem}
\author{Maxine Calle and Corey Dunn}

\begin{address}{CD: California State University at San Bernardino,
San Bernardino, CA 92407, USA. Email: \it
cmdunn@csusb.edu.}\end{address}
\begin{address}{MC: Reed College, Portland, OR 97202, USA. Email: \it callem@reed.edu. 
}\end{address}
\begin{abstract}

We algebraically compute all possible sectional curvature values for canonical algebraic curvature tensors, and use this result to give a method for constructing general sectional curvature bounds.  We use a well-known method to geometrically realize these results to produce a  hypersurface with prescribed sectional curvatures at a point.  By extending our methods, we give a relatively short proof of the Spectral Theorem for self-adjoint operators on a finite dimensional real vector space.

\end{abstract}
 \maketitle

\section{Introduction}  Let $V$ be a real vector space of finite dimension $n$, and let $V^* = {\rm Hom}(V, \mathbb{R})$ be the corresponding dual space.  An \emph{algebraic curvature tensor} $R \in \otimes^4 V^*$ satisfies the properties below for all $x, y, z, w \in V$:
\begin{equation}  \label{act}
\begin{array}{c}
R(x, y, z, w) = -R(y, x, z, w) = R(z, w, x, y),  {\rm\ and\ }\\
R(x, y, z, w) + R(x, w, y, z) + R(x, z, w, y) = 0.
\end{array}
\end{equation}
Let $\mathcal{A}(V)$ denote the set of all algebraic curvature tensors.  Let $S^2(V^*)$ denote the space of all symmetric bilinear forms over $V$, and let $\varphi \in S^2(V^*)$.  Define a \emph{canonical algebraic curvature tensor} $R_\varphi \in \mathcal{A}(V)$ as
$$
R_\varphi(x, y, z, w) = \varphi(x, w) \varphi(y, z) - \varphi(x, z) \varphi(y, w).
$$
If the vector space $V$ is endowed with a non-degenerate inner product $\langle\cdot,\cdot\rangle$, then the triple $(V, \langle\cdot, \cdot \rangle, R)$ is referred to as a \emph{model space}.  In this research, all inner products are assumed to be positive definite.  Throughout this work, we consider the vector space $V$ and a positive definite inner product $\langle\cdot, \cdot \rangle$ to be given, and so we do not always specifically reference them.  For the sake of convenience, we will refer to properties of a model space (such as sectional curvature) as a property of an algebraic curvature tensor when there is no possibility of confusion.

Let $Gr_2(V)$ be the Grassmannian of $2-$planes in $V$.  This space is compact and connected \cite{MS74}.  Given a model space $(V, \langle\cdot, \cdot \rangle, R)$ and a $2-$plane $\pi \in V$, define the \emph{sectional curvature} $\kappa(\pi)$ of $\pi$ to be 
$$
\frac{R(x, y, y, x)}{\langle x, x \rangle \langle y, y\rangle - \langle x, y \rangle^2},
$$
where $\pi = {\rm span}\{x, y\}$.  This quantity is independent of the basis chosen for $\pi$.

Interest in algebraic curvature tensors stems from basic results in differential geometry.  If $(M,g)$ is a smooth manifold  one may compute the Riemann curvature tensor $R_P$  at the point $P \in M$ using the Levi-Civita connection.  It is well known that $R_P$ satisfies the identities listed in Equation (\ref{act}), and thus $R_P$ is an algebraic curvature tensor on the tangent space $T_PM$.  Using the metric $g_P$ restricted to $T_PM$, $(T_PM, g_P, R_P)$ is a model space.    The converse is also true \cite{G07}:  given a model space $(V, \langle \cdot, \cdot\rangle, R)$, there exists a smooth manifold $(M,g)$,  a point $P \in M$, and a vector space isometry $\Psi: V \to T_PM$ so that $\Psi^*R_P = R$.  The process of finding such a manifold is generally referred to as a geometric realization.

A general line of questioning is to explore exactly what one can say about a manifold that is a geometric realization of any given model space(s).  In this way, the algebraic properties of the model space can influence the geometry of a geometric realization of it.  Although there are many examples of this \cite{BGN12, GKV02, G-S, KP94, T-V, TV86,  T05}, we give two examples relevant to our study.  The first example concerns the property of constant sectional curvature: up to local isometry, there is only one way to construct a manifold such that the model space at any of its points has constant sectional curvature.

 The second example begins with the fact  \cite{G07} that $\{R_\varphi | \varphi \in S^2(V^*)\}$ is spanning set for the space of algebraic curvature tensors.  For $R \in \mathcal{A}(V)$, denote 
 $$
 \nu(R) = \min\{k | R = \sum_{i = 1}^k \alpha_i R_{\varphi_i}\}, \quad \nu(n) = \max\{\nu(R) | R \in \mathcal{A}(V)\}.
 $$
 This notation seems to have originated in \cite{G07} and has since been studied by several authors.  The authors in \cite{DG04} show that $\nu(3) = 2$, and use the Nath Embedding Theorem \cite{N56} to prove $\nu(n) \leq \frac{n(n+1)}{2}$.  In this way, $\nu(R)$ functions as a lower bound of the codimension of any local embedding of a given manifold, where $R$ is the algebraic curvature tensor at a given point.  For this reason,  subsequent work \cite{DD10} has aimed to discover linear dependencies in the set of canonical algebraic curvature tensors.  
 
We have two major goals for this research.  First, we establish sharp bounds on the sectional curvature values of any canonical algebraic curvature tensor.  We interpret these bounds through a geometric realization result, and give a method for constructing bounds on the sectional curvature values of any algebraic curvature tensor.  By an extension of our methods (inspired by R. Klinger \cite{K91}), we meet our second goal: to give a short and self-contained proof of the Spectral Theorem. 

More specifically, after some preliminary comments, in  Section 2 we establish the following:

 \begin{theorem}  \label{sec}
Let $\varphi \in S^2(V^*)$, and let $\lambda_1, \ldots, \lambda_n$ be the eigenvalues of $\varphi$, repeated according to multiplicity.  Let $m$ and $M$, respectively, be the minimum and maximum of the set  $\{\lambda_i\lambda_j | i \neq j\}$. The set of sectional curvatures of $R_\varphi$ is precisely the interval $[m, M]$.
\end{theorem}

We use this result to prove two corollaries.  The first corollary (Corollary \ref{cor1}) uses a well-known result to  geometrically realize any interval as the set of sectional curvatures of a manifold at a point.  The manifolds we produce are hypersurfaces in Euclidean space.  The second corollary (Corollary \ref{cor2}) provides bounds (which are not sharp, see Remark \ref{notsharp}) on the set of sectional curvatures of an arbitrary algebraic curvature tensor $R$ in terms of $\nu(R)$ and the results from Theorem \ref{sec}.

In Section 3 we consider a canonical algebraic curvature tensor to provide a short and self-contained proof of the Spectral Theorem:

\begin{theorem}[The Spectral Theorem]  \label{spec}
Let $V$ be an inner product space.  If $\varphi \in S^2(V^*)$, then there exists an orthonormal basis $\{f_1, \ldots, f_n\}$ for $V$ for which $\varphi(f_i, f_j) = \lambda_{i} \delta_{ij}$.
\end{theorem}

\section{Sectional curvature bounds}  It is well known that given $\varphi \in S^2(V^*)$, there exists a self-adjoint linear map $A:V \to V$ characterized by the equation $\varphi(x, y) = \langle Ax, y\rangle$.  Relevant aspects of $\varphi$ are defined in terms of those same aspects in $A$, for example, $\ker(\varphi)$ and  ${\rm Rank}(\varphi)$ and defined as $\ker(A)$ and ${\rm Rank}(A)$, respectively.  

Similarly, eigenvalues and eigenvectors of a linear map $A: V \to V$ can related to the corresponding bilinear form $\varphi$.   It will be helpful to do so in reference to an orthonormal basis:  if $\{f_1, \ldots, f_n\}$ is an orthonormal basis for $V$, the $(j,i)$ matrix entry of $A$ on this basis is equal to $\varphi(f_i, f_j)$.  For this reason,   $f_i$ is an eigenvector with corresponding eigenvalue $\lambda_i$ if and only if $\varphi(f_i, f_j) = \lambda_i \delta_{ij}.$  

It is through this perspective that we establish our results.  We will repeatedly use the following lemma  whose proof uses a technique adapted from Klinger's work in \cite{K91}.

\begin{lemma}  \label{2space}
Let $\varphi \in S^2(V^*)$, and suppose $R_\varphi \neq 0$.  If $\pi$ is a 2-plane whose sectional curvature is extremal, then there exists an orthonormal basis of eigenvectors for $\pi$.
\end{lemma}

\begin{proof}
Since $R_\varphi \neq 0$, and $R_\varphi$ is determined by its sectional curvatures \cite{Lee97}, there exists a 2-plane $\pi$ of extremal nonzero sectional curvature.   Let $\{e_1, e_2\}$ be any orthonormal basis for $\pi$.  Let $\varphi|_{\pi}$ be the restriction of $\varphi$ to $\pi$. We now diagonalize\footnote{One could of course use the Spectral Theorem here, but since we use this result in our proof of the Spectral Theorem in the next section, we instead establish this directly.}  the symmetric form $\varphi|_{\pi}$ to create a new orthonormal basis $\{f_1, f_2\}$ for $\pi$.   For some $\theta$, set
$$
\begin{array}{r c l}
f_1 & = & \cos \theta e_1 - \sin \theta e_2, \\
f_2 & = & \sin\theta e_1 + \cos \theta e_2,
\end{array}
$$
where we will determine $\theta$ presently.  Let $\varphi_{ij} = \varphi(e_i, e_j)$.  We compute
$$\begin{array}{r c l}
\varphi(f_1, f_2) &=& \cos\theta \sin\theta (\varphi_{11} - \varphi_{22}) + (\cos^2\theta - \sin^2\theta) \varphi_{12} \\
 &= & \frac{1}{2}\sin(2\theta) (\varphi_{11} - \varphi_{22}) + \cos(2\theta) \varphi_{12}.
\end{array}$$
Choose $\theta$ so that $\varphi(f_1, f_2) = 0$.  Explicitly, if $\varphi_{12} = 0$ already, then $\theta = 0$.  Otherwise, 
$$
\theta = \frac{1}{2} \; {\rm arccot}\!\left(\frac{\varphi_{22} - \varphi_{11}}{2\varphi_{12}}\right).
$$
Now extend this basis for $\pi$ (in an arbitrary way) to form an orthonormal basis $\{f_1, f_2, f_3, \ldots, f_n\}$ for $V$. Note that on this basis, 
\begin{equation} \label{nonzero}
\begin{array}{r c l}
\varphi(f_1, f_2) &=& 0, {\rm\ and\ } \\
R_\varphi(f_1, f_2, f_2, f_1) &=& \varphi(f_1, f_1)\varphi(f_2, f_2)  \ {\rm is\ extremal\ and\ nonzero.}
\end{array}
\end{equation}

We now show that $\varphi(f_1, f_j) = \varphi(f_2, f_j) = 0$ for $j \geq 3$, which will complete the proof.  Choose any index $j \geq 3$, and  consider the plane 
$$\pi_\theta = {\rm span}\{\cos \theta f_1 + \sin \theta f_j, f_2\}.$$
Using a double angle formula and various curvature identities listed in the introduction, 
$$
\begin{array}{r c c}
\kappa(\pi_\theta) &=& \cos^2 \theta R_\varphi(f_1, f_2, f_2, f_1) + \sin^2 \theta R_\varphi(f_j, f_2, f_2, f_j) \\
 & & + \sin(2\theta)R_\varphi(f_1, f_2, f_2, f_j).
\end{array}$$
Since $\pi_0$ is extremal, $0$ is a critical point of the function $\theta \mapsto \kappa(\pi_\theta)$, so
$$
0 = \frac{d}{d\theta}[\kappa(\pi_\theta)]|_{\theta = 0}  =  2 R_\varphi(f_1, f_2, f_2, f_j).
$$
As a result, 
$$\begin{array}{r c l}
0 = R_\varphi(f_1, f_2, f_2, f_j) &=& \varphi(f_1, f_j)\varphi(f_2, f_2) - \varphi(f_1, f_2) \varphi(f_2, f_j) \\
& = &  \varphi(f_1, f_j)\varphi(f_2, f_2).
\end{array}$$
By Equation (\ref{nonzero}), we know $\varphi(f_2, f_2) \neq 0$, hence $\varphi(f_1, f_j) = 0$.  We can prove that $\varphi(f_2, f_j) = 0$ in a similar way by considering ${\rm span}\{f_1, \cos \theta f_2 + \sin \theta f_j\}$.
\end{proof}

We use Lemma \ref{2space} to establish Theorem \ref{sec}.  

\begin{proof}[\textbf{Proof of Theorem \ref{sec}}]  We start by proving that the maximal sectional curvature is $M$, the largest pairwise product of eigenvalues of $\varphi$.  Since $Gr_2(V)$ is compact, there exists a 2-plane $\Pi$ for which 
$$
\kappa(\Pi) = \sup\{\kappa(L) | L \in Gr_2(V)\}.
$$
We proceed in cases:  either $\kappa(\Pi) \neq 0$ or $\kappa(\Pi) = 0$.  Our goal is to show $\kappa(\Pi)$ to be a pairwise product of eigenvalues.  Thus, as the maximal sectional curvature value, $\kappa(\Pi) = M$.

If $\kappa(\Pi) \neq 0$, then we may use Lemma \ref{2space} to produce an orthonormal basis of eigenvectors $\{f_1, f_2\}$ for $\Pi$.  In this case,
$$
\kappa(\Pi) = R_\varphi(f_1, f_2, f_2, f_1) = \varphi(f_1, f_1)\varphi(f_2, f_2)
$$
is a product of eigenvalues and hence is equal to $M$.

Now suppose $\kappa(\Pi) = 0$.  Find an orthonormal basis $\{f_1, \ldots, f_n\}$ diagonalizing $\varphi$, and note  
$$
\kappa({\rm span}\{f_i, f_j\}) = R_\varphi(f_i, f_j, f_j, f_i) = \varphi(f_i, f_i)\varphi(f_j, f_j)
$$
is a product of eigenvalues of $\varphi$.  It is not possible that there are two nonzero eigenvalues of the same sign since their corresponding eigenvectors would span a 2-plane of positive sectional curvature which is contrary to our assumption that $\kappa(\Pi) = 0$ is the maximal sectional curvature.  Since $\dim(V) \geq 3$,  $0$ is an eigenvalue, and therefore $0 = \kappa(\Pi) = M$ is the largest pairwise product of eigenvalues as desired.

We now prove that the minimal sectional curvature is  $m$, the smallest pairwise product of eigenvalues of $\varphi$.  Proceeding similarly, there exists a 2-plane $\pi$ for which
$$
\kappa(\pi) = \inf\{\kappa(L) | L \in Gr_2(V)\}.
$$
We again proceed in cases:  either $\kappa(\pi) \neq 0$ or $\kappa(\pi) = 0$.  If $\kappa(\pi) \neq 0$ then use Lemma \ref{2space} once more to produce an orthonormal basis of eigenvectors  $\{f_1, f_2\}$ for $\pi$.  Then  
$$
\kappa(\pi) = R_\varphi(f_1, f_2, f_2, f_1) = \varphi(f_1, f_1) \varphi(f_2, f_2)
$$
is a product of eigenvalues and is hence equal to $m$.

Now suppose $\kappa(\pi) = 0$ and again diagonalize $\varphi|_{\pi}$ with the basis $\{f_1, f_2\}$ and extend it to an orthornormal basis $\{f_1, \ldots, f_n\}$ for $V$.  We then have
$$
0 = \kappa(\pi) = R_\varphi(f_1, f_2, f_2, f_1) = \varphi(f_1, f_1)\varphi(f_2, f_2).
$$
Thus, one of the two factors above is zero.  As exchanging the two vectors would not change the span nor disrupt the diagonalization, we may assume $\varphi(f_1, f_1) = 0$.  For $j \geq 3$, since $0$ is assumed to be the minimal sectional curvature, we find
$$\begin{array}{r c l}
0 \leq \kappa({\rm span}\{f_1, f_j\}) & =& R_\varphi(f_1, f_j, f_j, f_1)  \\
 &=& \varphi(f_1, f_1) \varphi(f_j, f_j) - \varphi(f_1, f_j)^2 \\
  & = & -\varphi(f_1, f_j)^2.
\end{array}$$
It follows that $\varphi(f_1, f_j) = 0$ for all $j \neq 1$, and so $f_1$ is an eigenvector corresponding to the eigenvalue $\varphi(f_1, f_1) = 0$.  Thus, $0$ is an eigenvalue of $\varphi$.  

To finish this portion of the proof, we prove that all other eigenvalues of $\varphi$ are either zero or have the same sign, which would show $0$ to be the smallest pairwise product of the eigenvalues.  To do this, find an orthonormal basis $\{e_1, \ldots, e_n\}$ for $V$ that diagonalizes $\varphi$.   Since 0 is minimal, we must have
$$
0 \leq \kappa({\rm span}\{e_i, e_j\}) = R_\varphi(e_i, e_j, e_j, e_i) = \varphi(e_i, e_i)\varphi(e_j, e_j).
$$
For this reason the nonzero eigenvalues of $\varphi$ cannot differ in sign, and this portion of the proof is complete.

We have proven that $\kappa(\Pi) = M$ and $\kappa(\pi) = m$ are, respectively, the maximal and minimal sectional curvatures.  Since $Gr_2(V)$ is connected and the mapping $L \mapsto \kappa(L)$ is continuous, the image of this mapping is connected as well, which completes the proof.
\end{proof}

We use a well known geometric realization result to establish the following corollary.  

\begin{corollary}  \label{cor1}
Let $[a,b]$ be any interval.  There exists a hypersurface $M$ in Euclidean space, a smooth metric $g$ on $M$, and a point $P\in M$ so that the set of sectional curvatures of $M$ at $P$ is precisely $[a,b]$.
\end{corollary}

\begin{proof}
Choose any collection of real numbers $\lambda_1, \ldots, \lambda_n \in \mathbb{R}$ (not necessarily distinct) whose minimal and maximal pairwise products are, respectively, $a$ and $b$.  Define the symmetric bilinear form $\varphi$ so that $\lambda_1, \ldots, \lambda_n$ are its eigenvalues.  One may now carry out the geometric realization procedure as outlined on page 74 of \cite{G01} to produce the desired result.
\end{proof}

Before proving our next corollary, we note that $R_{c\varphi} = c^2 R_\varphi$, and so any linear combination of canonical algebraic curvature tensors
$$
\sum \alpha_i R_{\varphi_i} = \sum \epsilon_i R_{\tilde \varphi_i},
$$
where $\epsilon_i = \pm1 = {\rm sign}(\alpha_i),$ and $\tilde \varphi_i = \sqrt{|\alpha_i|} \varphi_i$.

The following corollary finds bounds on the sectional curvature values of any algebraic curvature tensor and is a direct consequence of Theorem \ref{sec}.

\begin{corollary}  \label{cor2}
Let $R \in \mathcal{A}(V)$, and express $R = \sum_{i = 1}^{\nu(R)} \epsilon_i R_{\varphi_i}$, where $\epsilon_i = \pm 1$ and $\varphi_i \in S^2(V^*)$.      Let $M_i$ and $m_i$ be, respectively, the maximal and minimal pairwise products of the eigenvalues of $\varphi_i$, and let $M = \sum \epsilon_i M_i$ and $m = \sum \epsilon_i m_i$.  The set of possible sectional curvature values of the model space $(V, \langle \cdot, \cdot \rangle, R)$ is a closed subinterval of $[m,M]$. 
\end{corollary}

\begin{remark} \label{notsharp}{\rm
We give an example which shows that the bounds presented in Corollary \ref{cor2} are not sharp.  On an orthornormal basis $\{f_1, f_2, f_3\}$, define  $\varphi_1, \varphi_2 \in S^2(V^*)$ by having the following nonzero entries:
$$
\varphi_1(f_1, f_1) = \varphi_1(f_2, f_2) = \varphi_2(f_1, f_1) = \varphi_2(f_3, f_3) = 1.
$$
According to Theorem \ref{sec}, the maximal sectional curvature of both $R_{\varphi_1}$ and $R_{\varphi_2}$ is 1, so Corollary \ref{cor2} estimates the maximal sectional curvature of $R_{\varphi_1} + R_{\varphi_2}$ to be 2.  A straightforward calculation shows, however, that all sectional curvatures of $R_{\varphi_1} + R_{\varphi_2}$ are strictly less than 2. \hfill $\qedbox$
}\end{remark}

\section{The Spectral Theorem}

In this section we give a relatively short proof of the well known Spectral Theorem after establishing a helpful lemma.

\begin{lemma} \label{helpful}
Suppose $\varphi$ is a symmetric bilinear form on an inner product space $V$  with $\dim(V) \geq 2$.  There exists an orthonormal basis $\{f_1, e_2, \ldots, e_n\}$ for $V$ so that $f_1$ is an eigenvector of $\varphi$. 
\end{lemma}
\begin{proof}
Let $\varphi$ be given, and consider the algebraic curvature tensor $R_\varphi$.  We break the proof into cases:  either $R_\varphi \neq 0$ or $R_\varphi = 0$.

Suppose first that $R_\varphi \neq 0$.  Using Lemma \ref{2space}, we choose a 2-plane $\pi$ of extremal sectional curvature and find an orthonormal basis $\{f_1, e_2\}$ of eigenvectors for $\pi$, which we extend to the orthonormal basis $\{f_1, e_2, \ldots, e_n\}$ for $V$, establishing the result.

Now suppose $R_\varphi = 0$.  We claim that $0$ is an eigenvalue\footnote{In fact, it is known \cite{G07} that if $R_\varphi = 0$ then the rank of  $\varphi$ is less than or equal to 1, so that on a vector space of dimension 2 or more, there must be a nontrivial kernal of $\varphi$.  However, we wish to give a self-contained and alternative proof here as part of our effort to produce a new proof of the Spectral Theorem.} of $\varphi$.  In this case, any corresponding unit eigenvector may be chosen first as part of an orthonormal basis for $V$, which would again establish the result.  

Suppose to the contrary that $0$ is not an eigenvalue of $\varphi$.  Therefore, for any orthonormal basis  $\{f_1, \ldots, f_n\}$, the matrix with $\varphi(f_j, f_i)$ as its $(i,j)-$entry must have a nonzero determinant.  By expanding this determinant by cofactors, there must be at least one nonzero $2\times2$ minor.  Thus for some indices $i_1, i_2, j_1, j_2$,
$$
\begin{array}{r c l}
0 & \neq  & \varphi(f_{i_1}, f_{j_1}) \varphi(f_{i_2}, f_{j_2}) - \varphi(f_{i_1}, f_{j_2}) \varphi(f_{i_2}, f_{j_1}) \\
 & = & R_\varphi(f_{i_1}, f_{i_2}, f_{j_2}, f_{j_1}),
\end{array}
$$
which contradicts the assumption that $R_\varphi = 0$.
\end{proof}

With these short preliminary facts established, we can now give a new proof of The Spectral Theorem. 

\begin{proof}[\textbf{Proof of The Spectral Theorem}]  The result is trivial if $\dim(V) = 1$, so assume $\dim(V) \geq 2$.  By Lemma \ref{helpful}, there exists an orthonormal basis $\{f_1, e_2, \ldots, e_n\}$ for $V$ so that $f_1$ is an eigenvector of $\varphi$.  Now consider the vector space $W = f_1^\perp = {\rm span}\{e_2, \ldots, e_n\}$.  Since our inner product is positive definite, the restriction of the inner product to $W$ remains positive definite.  In addition, the restricted curvature tensor satisfies $(R_\varphi)|_W = R_{\varphi|_W}$.
   
 We first consider when $\dim(W) \geq 2$.  We again use Lemma \ref{helpful} to find an orthonormal basis $\{f_2, \tilde e_3, \ldots, \tilde e_n\}$ for $W$ so that $f_2$ is an eigenvector of $\varphi|_W$.  That is,  
   $$
  \varphi|_W(f_2, \tilde e_k) = \varphi(f_2, \tilde e_k) = 0.
  $$
 Since $f_1$ is an eigenvector of $\varphi$, we also have $\varphi(f_1, f_2) = 0$, so $f_2$ is also an eigenvector of $\varphi$.  We have a new orthonormal basis $\{f_1, f_2, \tilde e_3, \ldots, \tilde e_n\}$ for $V$, where both $f_1$ and $f_2$ are eigenvectors of $\varphi$.

Now consider $W_2 = {\rm span}\{f_1, f_2\}^\perp$.  If $\dim(W_2) > 1$, then repeat the process outlined above, and continue until there is an orthonormal basis $\{f_1, \ldots, f_{n-1}, \bar e_n\}$ for $V$ (this is the case above if $\dim(W) = 1$ above), and for any $i \neq j$ inclusively between $1$ and $n-1$ we have
 $$
 \varphi(f_i, f_j) = \varphi(f_i, \bar e_n) = 0.
 $$
 This demonstrates that $f_n = \bar e_n$ is also an eigenvector, completing the proof.   \end{proof}

\section{Acknowledgments}  The authors appreciate helpful discussions with I. Stavrov concerning this work.  This research was jointly funded by California State University, San Bernardino, and the NSF grant DMS-1758020.

\end{document}